\documentclass[11pt]{amsart}

\usepackage{palatino} 
\usepackage{hyperref}
\usepackage{tikz}
\usetikzlibrary{positioning}

\usepackage{amsmath,amssymb,amsthm,verbatim, amsfonts,amscd,flafter,epsf, epsfig,graphicx,verbatim,pinlabel,mathrsfs}
\usepackage[all]{xy}
\usepackage{epsf}
\usepackage[abs]{overpic}
\usepackage{epstopdf}

\usepackage{xcolor}
\usepackage{mathtools}
\usepackage{scalerel}

\definecolor{bred}{rgb}{1,0,.2}
\definecolor{blue}{rgb}{0,0,1}

\newtheorem{theorem}{Theorem}[section]

\newtheorem{proposition}[theorem]{Proposition}

\theoremstyle{definition}
\newtheorem{definition}[theorem]{Definition}

\def\leave#1{{}}

\title{Common positive stabilisation of open book decompositions}
\author[Joan Licata]{Joan Licata}
\address{Mathematical Sciences Institute, Australian National University}
\address{France-Australia Mathematical Sciences and Interactions ANU-CNRS International Research Laboratory}
\email{joan.licata@anu.edu.au}

\author[Vera V\'ertesi]{Vera V\'ertesi}
\address{University of Vienna}
\email{vera.vertesi@univie.ac.at}

\begin{document}

\maketitle

\begin{abstract} 

The Giroux Correspondence states that two open book decompositions supporting the same contact structure are related by a sequence of positive open book stabilisations and destabilisations.  In this note we show that any two open book decompositions supporting isotopic contact structures admit a common positive stabilisation.
\end{abstract}

\section{Introduction}

The Giroux Correspondence is a powerful tool in contact geometry, classifying which topological \textit{open book decompositions} support equivalent \textit{contact structures} \cite{Giob}.  Specifically, \textit{positive stabilisation} is a move that replaces one open book decomposition with another, and the Giroux Correspondence establishes that this move and its inverse (\textit{positive destabilisation}) suffice to relate any two open book decompositions which support isotopic contact structures.  In this note, we clarify that when  some pair of open books is related by a sequence of mixed stabilisations and destabiliations, then  this connecting sequence may be replaced a sequence consisting first of stabilisations and finally of destabilisations.  

\begin{theorem}\label{thm:cps} Suppose that $(\Sigma_1, \phi_1)$ and $(\Sigma_n, \phi_n)$ are abstract open book decompositions related by a sequence of positive  stabilisations and destabilisations.  Then there is an open book decomposition $(\Sigma, \phi)$ which is a positive stabilisation of $(\Sigma_i, \phi_i)$ for $i=1,n$.  
\end{theorem}

This yields the following more precise version of the Giroux Correspondence:

\begin{theorem}[Giroux Correspondence]Suppose that $(\Sigma_1, \phi_1)$ and $(\Sigma_n, \phi_n)$ are abstract open book decompositions supporting equivalent contact structures.  Then  $(\Sigma_1, \phi_1)$ and $(\Sigma_n, \phi_n)$ admit a common positive stabilisation.
\end{theorem}

While the proof of Theorem~\ref{thm:cps} is not difficult, the distinction between arbitrary sequences of (de)stabilisations and the existence of a common positive stabilisation has not been made clearly in the literature.  (A preprint version of one of our papers is among the guilty parties.)  As we observe next, the stronger claim allows for an interesting comparison to the case of Heegaard splittings on a topological $3$-manifold.  

The Reidemeister-Singer Theorem states that any two Heegaard splittings of $M^3$ become isotopic after finitely many stabilisations.  This structure theorem sets the stage for more nuanced questions about how the topology of the manifold affects the number of stabilisations required, with both positive and negative results known.  For example, in broad classes of manifolds, any two Heegaard splittings of the same genus become isotopic with a single stabilisation \cite{Scharl} \cite{Schult} \cite{Sedg}.   In contrast, for any $n$ there exists a manifold with a pair of splittings that remain distinct after $n$ stabilisations \cite{HTT}. Open book stabilisation is more subtle than Heegaard splitting stabilisation, but analogous questions both exist and remain largely open.


\subsection{Acknowledgements}
The authors would like to thank John Etnyre and Eric Stenhede for interesting conversations on this topic.  This research was supported in part by the Austrian Science Fund (FWF) P 34318. For open access purposes, the second author has applied a CC BY public copyright license to any author-accepted manuscript version arising from this submission.

\section{Proof of Theorem~\ref{thm:cps}}
We find the Giroux Correspondence more satisfying when stated in terms of embedded open books, but abstract open books are easier to work with in the present case and we will use this language exclusively. Background and context for open books may be found in \cite{Etnyre}.

\begin{definition}
An \textit{abstract open book} is a pair $(\Sigma, \phi)$, where $\Sigma$ is a surface with non-empty boundary and $\phi$ is a mapping class of $\Sigma$.  
\end{definition}

We define stabilisation as a formal operation on the pair $(\Sigma, \phi)$.

\begin{definition}\label{def:stab}  Let $\alpha$ be a properly embedded arc in $\Sigma$.  The \textit{stabilisation along} $\alpha$ is the open book  $(\Sigma\cup H_\alpha, \phi_\alpha \circ \tau_{{\alpha}})$, where
\begin{itemize}
\item $H_\alpha$ is a one-handle with attaching sphere $\partial \alpha \subset \partial \Sigma$;
\item $\tau_{{\alpha}}$ denotes a positive Dehn twist around the simple closed curve formed by $\alpha$ and the core of $H_\alpha$; and 
\item $\phi_\alpha$ denotes the extension of $\phi$ by the identity on $H_\alpha$.  
\end{itemize}

If an open book has the form $(\Sigma\cup H_\alpha, \phi_\alpha \circ \tau_{{\alpha}})$ for some $\alpha \subset \Sigma$,  \textit{destabilisation} replaces it by $(\Sigma, \phi)$.
\end{definition}

The term \textit{stabilisition} denotes stabilisation along some unspecified arc $\alpha$.

To prove the theorem, it suffices to replace an arbitrary sequence of open books of the form

\[ (\Sigma_i, \phi_i)\xrightarrow{\text{destab}} (\Sigma_{i+1}, \phi_{i+1})\xrightarrow{\text{stab}} (\Sigma_{i+2}, \phi_{i+2})\]
 with a new sequence of the form
\[ (\Sigma_i, \phi_i)\xrightarrow{\text{stab}} (\Sigma'_{i+1}, \phi'_{i+1})\xrightarrow{\text{destab}} (\Sigma_{i+2}, \phi_{i+2}).\]

Proposition~\ref{prop:cps}, below,  establishes that this is possible.  Accepting this result, consider any sequence $(\Sigma_i, \phi_i)$ of positive stabilisations and destabilisations connecting $(\Sigma_1, \phi_1)$ to $(\Sigma_n, \phi_n)$.  Each open book decomposition has a natural notion of complexity provided by the negative Euler characteristic of $\Sigma$. Stabilisation increases this complexity by one, while destabilisation decreases it by one.  Thus, we may consider the sequence of complexities associated to the given sequence of open books.

Proposition~\ref{prop:cps} shows that an open book $(\Sigma_{i}, \phi_{i})$ in the interior of the sequence where the complexity function attains a local minimum $j$ may be replaced by $(\Sigma'_{i}, \phi'_{i})$ with complexity $j+2$.  After a finite number of such substitutions, the sequence of complexities will have local minima only at the initial and final open books and a single local maximum in the interior.  This local maximum is associated to the desired positive stabilisation of both endpoints. 

This proof provides an algorithm for replacing an arbitrary sequence of open books with a sequence that has a single open book of maximal complexity, and the Euler characteristic of the page of this common positive stabilisation is determined by the relative Euler characteristics of the initial open books and the length of the sequence connecting them.  However, one should not expect this to be the minimal common positive stabilisation, as an inefficient starting sequence will produce a more highly stabilised peak.  

Finally, we state and prove the promised Proposition~\ref{prop:cps}.

\begin{proposition}\label{prop:cps} Suppose that $\alpha, \beta$ are properly embedded arcs on $\Sigma$.  Then $(\Sigma\cup H_\alpha \cup H_\beta, \phi_{\alpha \cup \beta} \circ \tau_\alpha \circ \tau_\beta)$ is a stabilisation of both $(\Sigma\cup H_\alpha, \phi_{\alpha} \circ \tau_\alpha)$ and $(\Sigma\cup H_\beta, \phi_\beta \circ \tau_\beta)$.
\end{proposition}

\begin{proof} The first claim is immediate, as  $(\Sigma\cup H_\alpha \cup H_\beta, \phi_{\alpha \cup \beta} \circ \tau_\alpha \circ \tau_\beta)$ is the result of stabilising $(\Sigma\cup H_\alpha, \phi_{\alpha} \circ \tau_\alpha)$ along $\beta$ in $\Sigma \cup H_\alpha$.   It remains to show that $(\Sigma\cup H_\alpha \cup H_\beta, \phi_{\alpha \cup \beta} \circ \tau_\alpha \circ \tau_\beta)$ is also a stabilisation of $(\Sigma\cup H_\beta, \phi_\beta \circ \tau_\beta)$. To do so, we use elementary properties of the mapping class group.

Suppose that $a$ is a simple closed curve on a surface and let $f$ be a diffeomorphism of the surface.  Then the following equality of mapping classes holds: $f^{-1}\tau_{f(a)} = \tau_af^{-1}$. See, for example, Fact 3.7 in \cite{FM}. 

Now apply the above relation to the surface $\Sigma\cup H_\alpha \cup H_\beta$, setting $f=\tau^{-1}_{{\beta}}$ and replacing $a$ with the simple closed curve formed by $\alpha$ and the core of $H_\alpha$.  This yields
\[  \tau_\beta \tau_{\tau^{-1}_\beta(\alpha)} = \tau_\alpha\tau_\beta.\]
It follows that $(\Sigma\cup H_\alpha \cup H_\beta, \phi_{\alpha \cup \beta} \circ \tau_\alpha \circ \tau_\beta)=(\Sigma\cup H_\alpha \cup H_\beta, \phi_{\alpha \cup \beta} \circ \tau_\beta \circ\tau_{\tau^{-1}_\beta(\alpha)})$.  Thus, this open book is the result of stabilising $(\Sigma\cup H_\beta, \phi_{\beta} \circ \tau_\beta)$ along $\tau^{-1}_\beta(\alpha)$ in $\Sigma \cup H_\beta$.  
\end{proof}

%

\bibliographystyle{alpha}
\bibliography{fob}
\end{document}